\documentclass[12pt]{amsart}

\usepackage{bm}
\usepackage{fullpage}
\usepackage{amssymb}
\usepackage{amsmath}
\usepackage{color}

\newtheorem{lemma}{Lemma}

\newtheorem{proposition}{Proposition}
\newtheorem{theorem}{Theorem}
\newtheorem{corollary}{Corollary}
\newtheorem{conjecture}{Conjecture}

\theoremstyle{remark}

\def\ds{\displaystyle}

\def\Z{\mathbb{Z}}

\def\zero{{\bf 0}}

\def\x{\boldsymbol{x}}
\def\y{\boldsymbol{y}}
\def\z{\boldsymbol{z}}

\def\alt{\operatorname{alt}}

\author{Meysam Alishahi}
\address{M. Alishahi,
School of Mathematical Sciences,
Shahrood University of Technology, Shahrood, Iran}
\email{meysam\_alishahi@shahroodut.ac.ir}

\author{Fr\'ed\'eric Meunier}
\address{F. Meunier, Universit\'e Paris Est, CERMICS, 77455 Marne-la-Vall\'ee CEDEX, France}
\email{frederic.meunier@enpc.fr}

\title{Fair splitting of colored paths}

\begin{document}

\maketitle

\begin{abstract}
This paper deals with two problems about  splitting fairly a path with colored vertices, where ``fairly'' means that each part contains almost the same amount of vertices in each color. 

Our first result states that it is possible to remove one vertex per color from a path with colored vertices so that the remaining vertices can be fairly split into two independent sets of the path. 
It implies in particular a conjecture of Ron Aharoni and coauthors. The proof uses the octahedral Tucker lemma.

Our second result is the proof of a particular case of a conjecture of D{\"o}m{\"o}t{\"o}r P{\'a}lv{\"o}lgyi about fair splittings of necklaces for which one can decide which thieves are advantaged. The proof is based on a rounding technique introduced by Noga Alon and coauthors to prove the discrete splitting necklace theorem from the continuous one.
\end{abstract}

\section{Introduction}

This paper is about fair splittings of paths with colored vertices. ``Fair'' means throughout the paper that for each color $j$, the numbers of vertices of color $j$ in each part differ by at most one.

Given a path $P$ whose vertex set is partitioned into $m$ subsets $V_1,\ldots,V_m$, Aharoni et al.~\cite[Conjecture 1.6]{aharoni2016fair} conjectured that there always exists an independent set $S$ of $P$ such that $|S\cap V_j|\geq |V_j|/2-1$, with strict inequality holding for at least $m/2$ subsets $V_j$. 
 We prove that we can actually remove one vertex from each $V_j$ so that the remaining vertices can be fairly split into two independent sets of $P$ of almost same size. 

 \begin{theorem}\label{main}
Given a path $P$ whose vertex set is partitioned into $m$ subsets $V_1,\ldots,V_m$, there always exist two disjoint independent sets $S_1$ and $S_2$ covering all vertices but one in each $V_j$, with sizes differing by at most one, and satisfying for each $i\in\{1,2\}$
$$ |S_i\cap V_j| \geq  \frac{|V_j|} 2 -1\qquad\mbox{for all $j\in[m]$.}$$
\end{theorem}

 Theorem~\ref{main} implies in particular that the conjecture by Aharoni et al. is true: the equality $|S_1\cap V_j|+|S_2\cap V_j|= |V_j|-1$ holds for every $j$ and thus one of $S_1$ or $S_2$ satisfies the inequality strictly for at least $m/2$ indices $j$.

The Borsuk-Ulam theorem was originally used for proving a special case of their conjecture (Theorem 1.7 of their paper). Here, we use the octahedral Tucker lemma, which is a combinatorial version of the Borsuk-Ulam theorem. \\

Another result in combinatorics deals with the fair splitting of a path whose vertices are colored and is a consequence of the Borsuk-Ulam theorem: the splitting necklace theorem. Consider an open necklace of $n$ beads, each having a color $j\in[m]$. We denote by $a_j$ the number of beads of color $j$. A {\em fair $q$-splitting} is the partition of the beads into $q$ parts, each containing $\lfloor a_j/q\rfloor$ or $\lceil a_j/q\rceil$ beads of color $j$. The picturesque motivation is the division of the necklace among $q$ thieves after its robbery. A theorem of Alon~\cite{alon1987splitting} states that such a partition is always achievable with no more than $(q-1)m$ cuts when $a_j$ is divisible by $q$. Using a flow-based rounding argument, Alon, Moshkovitz, and Safra~\cite{alon2006algorithmic} were able to show that such a partition is achievable without this assumption. 
In a paper in which he proved the splitting necklace theorem when $q=2$ via the octahedral Tucker lemma, P{\'a}lv{\"o}lgyi~\cite{palvolgyi2009combinatorial} conjectured that for each $j$ such that $a_j$ is not divisible by $q$, it is possible to decide which thieves get $\lfloor a_j/q\rfloor$ and which get $\lceil a_j/q\rceil$ and to still have a fair $q$-splitting not requiring more than $(q-1)m$ cuts. The conjecture is known to be true when $q=2$ (see~\cite{palvolgyi2009combinatorial}), when $a_j\leq q$ for all $j$ (by a greedy assignment), and when $m=2$ (see~\cite{meunier2014simplotopal}).
With a simple trick inspired by the argument of Alon, Moshkovitz, and Safra, we show that the conjecture is true when the remainder in the euclidian division of $a_j$ by $q$ is $0$, $1$, or $q-1$ for all $j$. This result implies in particular the conjecture for $q=3$.

\section{Fair splitting by independent sets of a path}



\subsection{Proof}

The combinatorial counterpart of the Borsuk-Ulam theorem is Tucker's lemma. Our main tool is a special case of this counterpart when the triangulation is the first barycentric subdivision of the cross-polytope. It turns out that in this case, Tucker's lemma can be directly expressed in combinatorial terms. This kind of formulation goes back to Matou\v{s}ek~\cite{matouvsek2004combinatorial} and Ziegler~\cite{ziegler2002generalized}.

As in oriented matroid theory, we define $\preceq$ to be the following partial order on $\{+,-,0\}$: $$0\preceq +,\quad 0\preceq -,\quad+\mbox{ and }-\mbox{ are not comparable.}$$ We extend it for sign vectors by simply taking the product order: for $\x,\y\in\{+,-,0\}^n$, we have $\x\preceq\y$ if the following implication holds for every $i\in[n]$
$$x_i\neq 0\Longrightarrow x_i=y_i.$$ 


\begin{lemma}[Octahedral Tucker lemma]\label{lem:oct_tucker}
Let $s$ and $n$ be positive integers. If there exists a map $\lambda:\{+,-,0\}^n\setminus \{\zero\} \rightarrow\{\pm 1,\ldots,\pm s\}$
 satisfying $\lambda(-\x)=-\lambda(\x)$ for all $\x$ and $\lambda(\x)+\lambda(\y)\neq 0$ when $\x\preceq\y$, then $s\geq n$.
\end{lemma}

In the proof, we use the following notations for $\x\in\{+,-,0\}^n$:
$$\x^+=\{i\in[n]\colon x_i=+\}\qquad\mbox{and}\qquad\x^-=\{i\in[n]\colon x_i=-\}.$$  Note that $\x\preceq\y$ if and only if simultaneously $\x^+\subseteq\y^+$ and $\x^-\subseteq\y^-$. We also use the notion of alternating sequences. A sequence of elements in $\{+,-,0\}^n$ is {\em alternating} if all terms are nonzero and any two consecutive terms are different. Given an $\x=(x_1,\ldots,x_n)\in\{+,-,0\}^n$, we denote by $\alt(\x)$ the maximum length of an alternating subsequence of $x_1,\ldots,x_n$.

\begin{proof}[Proof of Theorem~\ref{main}]
The proof consists in applying Lemma~\ref{lem:oct_tucker} on a map $\lambda$ we define now. Let $n$ be the number of vertices of $P$. Without loss of generality, we assume that the vertices of $P$ are $1,2,\ldots,n$ in this order when one goes from one endpoint to the other. In the definition of $\lambda$, we use the quantity $t=\max\big\{\alt(\x)\colon \x\in\{+,-,0\}^n\;\,\mbox{s.t.}\;J(\x)=\varnothing\big\}$, where $$J(\x)=\Big\{j\in[m]\colon |\x^+\cap V_j|=|\x^-\cap V_j|=|V_j|/2\quad\mbox{or}\quad\max(|\x^+\cap V_j|,|\x^-\cap V_j|)>|V_j|/2\Big\}.$$ Note that $t\geq 0$.

Consider a vector $\x\in\{+,-,0\}^n\setminus\{\zero\}$. We distinguish two cases.
In the case where $J(\x)\neq\varnothing$, we set  $\lambda(\x)=\pm (t+j')$, where $j'$ is the maximum element in $J(\x)$ and where the sign is defined as follows. When $|\x^+\cap V_{j'}|=|\x^-\cap V_{j'}|=|V_{j'}|/2$, the sign is $+$ if $\min(\x^+\cap V_{j'})<\min(\x^-\cap V_{j'})$ and $-$ otherwise. When $\max\big(|\x^+\cap V_{j'}|,|\x^-\cap V_{j'}|\big)>|V_{j'}|/2$, the sign is $+$ if $|\x^+\cap V_{j'}|>|V_{j'}|/2$, and $-$ otherwise.
In the case where $J(\x)=\varnothing$, we set $\lambda(\x)=\pm\alt(\x)$, where the sign is the first nonzero element of $\x$.

Let us check that the map $\lambda$ satisfies the condition of Lemma~\ref{lem:oct_tucker}. Consider $\x\in\{+,-,0\}^n\setminus\{\zero\}$. The relation $J(-\x)=J(\x)$ immediately implies $\lambda(-\x)=-\lambda(\x)$. Now, consider $\x,\y\in\{+,-,0\}^n\setminus\{\zero\}$ such that $\x\preceq\y$ and $|\lambda(\x)|=|\lambda(\y)|$. We cannot have simultaneously $J(\x)=\varnothing$ and $J(\y)\neq\varnothing$ since otherwise $|\lambda(\x)|\leq t$ and $|\lambda(\y)|>t$. Suppose first that $J(\x)\neq\varnothing$. Since $\x^+\subseteq\y^+$ and $\x^-\subseteq\y^-$, the signs of $\lambda(\x)$ and $\lambda(\y)$ are the same. Suppose now $J(\y)=\varnothing$. Then $J(\x)=\varnothing$. In this case we have $\alt(\x)=\alt(\y)$, and it is simple to check (and well-known) that the first nonzero elements of $\x$ and $\y$ are the same.

We can thus apply Lemma~\ref{lem:oct_tucker} with $s=t+m$. It gives $t+m\geq n$, which implies that there exists $\z'\in\{+,-,0\}^n$ such that $J(\z')=\varnothing$ and $\alt(\z')\geq n-m$, which in turn implies that there exists $\z\in\{+,-,0\}^n$ such that $J(\z)=\varnothing$ and $\alt(\z)=|\z^+|+|\z^-|=n-m$. 
Let $S_1=\z^+$ and $S_2=\z^-$. They are both independent sets of $P$ and their sizes differ by at most one. 
Because $J(\z)=\varnothing$, we have $|S_1\cap V_j|+|S_2\cap V_j|\leq |V_j|-1$ for all $j$. The fact that $|S_1|+|S_2|=n-m$ leads then to $|S_1\cap V_j|+|S_2\cap V_j|=|V_j|-1$ for all $j$. Now, using again $J(\z)=\varnothing$, we have each of $|S_1\cap V_j|$ and $|S_2\cap V_j|$ non-larger than $|V_j|/2$, which leads directly to the inequality of the statement.
\end{proof}

The proof shows that $S_1$ and $S_2$ alternate along the path $P$. Theorem~\ref{main} combined with this remark leads to the following corollary, 
which improves Theorem~1.8 in the aforementioned paper by Aharoni et al. 
In particular, if $m$ and the number of vertices of $P$ have the same parity, replacing ``path'' by ``cycle'' in the statement of Theorem~\ref{main} does not change the conclusion.

\begin{corollary}\label{cor}
Given an $n$-cycle $C$ whose vertex set is partitioned into $m$ subsets $V_1,\ldots,V_m$, there always exist two disjoint sets $S_1$ and $S_2$ covering all vertices but one in each $V_j$, with sizes differing by at most one, and satisfying for each $i\in\{1,2\}$
$$ |S_i\cap V_j| \geq  \frac{|V_j|} 2 -1\qquad\mbox{for all $j\in[m]$,}$$
where one of the $S_i$'s is an independent set of size $\left\lfloor \frac{n-m} 2\right\rfloor$ and the other induces at most $\left\lceil \frac{n-m} 2\right\rceil- \left\lfloor \frac{n-m} 2\right\rfloor$ edge. 
\end{corollary}

\begin{proof}
Remove an arbitrary edge of the cycle $C$ and apply Theorem~\ref{main} to the path $P$ we obtain this way. When $n-m$ is even, the fact that $S_1$ and $S_2$ alternate implies that there are independent sets of $C$ as well, with the same property as for the case of a path. When $n-m$ is odd, one of $S_1$ and $S_2$ is independent and of size $\left\lfloor\frac{n-m} 2\right\rfloor$ and the other is of size $\left\lceil \frac{n-m} 2\right\rceil$ and may contain the two endpoints of $P$, but the other properties are kept.
\end{proof}



\subsection{Extension to arbitrary numbers of independent sets}

We conjecture that Theorem~\ref{main} can be extended for any number of independent sets. In a graph, a subset of vertices is {\em $q$-stable} if no two of them are at distance less than $q$, where the distance is counted in terms of edges. In particular, the $2$-stable sets of a graph are precisely its independent sets.

\begin{conjecture}\label{conj}
Given a positive integer $q$ and a path $P$ whose vertex set is partitioned into $m$ subsets $V_1,\ldots,V_m$ of sizes at least $q-1$,
there always exist pairwise disjoint $q$-stable sets $S_1,\ldots,S_q$ covering all vertices but $q-1$ in each $V_j$, with sizes differing by at most one, and satisfying
$$ |S_i\cap V_j|\geq \left\lfloor \frac{|V_j|+1} q\right\rfloor-1$$ for all $i\in[q]$ and all $j\in[m]$.
\end{conjecture}


 
This conjecture is obviously true for $q=1$ and Theorem~\ref{main} is the special case where $q=2$. The way Theorem~\ref{main} is written might suggest a lower bound of the form $|V_j|/q-1$, but there are simple counterexamples. Consider for instance a case with $q=3$ and $|V_1|=7$. If there were pairwise disjoint 3-stable sets $S_1,S_2,S_3$ covering all vertices of $V_1$ but two, with each $|S_i\cap V_1|$ of size at least $7/3-1=1.33$, we would have $|V_1|-2= |S_1\cap V_1|+|S_2\cap V_1|+|S_3\cap V_1|\geq 6$, a contradiction.

The independent sets $S_i$ of Theorem~\ref{main}  satisfy automatically the additional inequality $|S_i\cap V_j|\leq |V_j|/2$ for all $j$, and it is actually used in the proof itself of that theorem. 
We believe that the stronger version of Conjecture~\ref{conj} with an upper bound of $|V_j|/q$ on $|S_i\cap V_j|$ for all $i$ and $j$ is true.

Since Conjecture~\ref{conj} is true for $q\in\{1,2\}$, the following proposition implies that it is true for any power of two. The other cases remain open.

\begin{proposition}\label{prop}
If Conjecture~\ref{conj} holds for both $q'$ and $q''$, then it holds also for $q=q'q''$.
\end{proposition}

The proof uses extensively the relations 
\begin{equation}\label{eq}
\left\lfloor\frac 1 c\left\lfloor \frac a b\right\rfloor\right\rfloor=\left\lfloor\frac{a}{bc}\right\rfloor\qquad\mbox{and}\qquad\left\lceil\frac 1 c{\left\lceil \frac a b\right\rceil}\right\rceil=\left\lceil\frac{a}{bc}\right\rceil
\end{equation} that hold for any $a,b,c\in\Z$ (actually, only $c\in\Z$ is required for these relations to hold).

Let us prove the left one. We have $\left\lfloor\frac 1 c\left\lfloor \frac a b\right\rfloor\right\rfloor\leq \frac 1 c\left\lfloor\frac a b\right\rfloor \leq\frac a {bc}$, and thus $\left\lfloor\frac 1 c\left\lfloor \frac a b\right\rfloor\right\rfloor\leq\left\lfloor\frac{a}{bc}\right\rfloor$. We also have $\frac a {bc}\geq \left\lfloor\frac{a}{bc}\right\rfloor$ and thus $\frac a {b}\geq c\left\lfloor\frac{a}{bc}\right\rfloor$, which implies since $c\in\Z$ that $\left\lfloor\frac a {b}\right\rfloor\geq c\left\lfloor\frac{a}{bc}\right\rfloor$. Therefore $\left\lfloor\frac 1 c\left\lfloor \frac a b\right\rfloor\right\rfloor\geq\left\lfloor\frac{a}{bc}\right\rfloor$. The right equality is proved similarly.

\begin{proof}[Proof of Proposition~\ref{prop}]

Consider a path $P$ whose vertices are partitioned into $m$ subsets $V_1,\ldots,V_m$ of sizes at least $q-1$. We assume that the conjecture is true for $q'$ and $q''$. We aim at proving that there exist pairwise disjoint $q$-stable sets $S_1,\ldots,S_q$ satisfying the conclusion of Conjecture~\ref{conj}. 

 Since the conjecture is assumed to be true for $q'$, there exist pairwise disjoint $q'$-stable sets $T_1,\ldots,T_{q'}$ of $P$, covering all vertices but $q'-1$ of each $V_j$ and such that for each $i\in[q']$
$$\begin{array}{l}
\ds{|T_i\cap V_j| \geq \left\lfloor  \frac{|V_j|+1} {q'} \right\rfloor-1}\quad\mbox{for all $j\in[m]$,}\quad\mbox{ and}\\ \\
\ds{\left\lceil \frac{n-(q'-1)m} {q'}\right\rceil \geq |T_i| \geq  \left\lfloor \frac{n-(q'-1)m} {q'}\right\rfloor,}
\end{array}$$ where $n$ is the number of vertices of $P$.
The conjecture being also assumed to be true for $q''$, we apply it on the path $Q_i$ whose vertices are the elements of $T_i$ (in the relative positions they have on $P$), for each $i\in[q']$. Note that $|T_i\cap V_j|\geq q''-1$ for every $j$. Therefore, for each $i$, there are pairwise disjoint $q''$-stable sets $S_{i1},\ldots,S_{iq''}$ of $Q_i$ covering all vertices but $q''-1$ of each $T_i\cap V_j$ and such that for each $k\in[q'']$ 
$$\begin{array}{l}
\ds{|S_{ik}\cap V_j| \geq \left\lfloor \frac{|T_i\cap V_j|+1} {q''} \right\rfloor-1}\quad\mbox{for all $j\in[m]$,}\quad\mbox{ and}\\ \\
\ds{\left\lceil \frac{|T_i|-(q''-1)m} {q''}\right\rceil \geq |S_{ik}| \geq  \left\lfloor \frac{|T_i|-(q''-1)m} {q''}\right\rfloor.}
\end{array}$$
Using the relation~\eqref{eq}, we get directly that each of the $q=q'q''$ subsets $S_{ik}$ satisfies
$$\begin{array}{l}
\ds{ |S_{ik}\cap V_j| \geq\left\lfloor \frac{|V_j|+1} q \right\rfloor-1}\quad\mbox{for all $j\in[m]$,}\quad\mbox{and}\\ \\
\ds{\left\lceil \frac{n-(q-1)m} q\right\rceil \geq |S_{ik}| \geq  \left\lfloor \frac{n-(q-1)m} q\right\rfloor.}
\end{array}$$
For each $V_j$, the number of uncovered vertices is exactly $q'-1+q'(q''-1)=q-1$. Moreover, each $T_i$ is $q'$-stable for $P$ and each $S_{ik}$ is $q''$-stable for $Q_i$. Thus each $S_{ik}$ is $q$-stable for $P$ and this finishes the proof.
\end{proof}

The proof of Proposition~\ref{prop} can be adapted in a straightforward way to get that it is also true for the aforementioned version of Conjecture~\ref{conj} with the upper bounds on the $|S_i\cap V_j|$'s. We get thus that this stronger conjecture is also true for any power of two. 

\section{Fair splitting of the necklace with advantages}

The result we prove in this section is the following one. We denote by $r_j$ the remainder of the euclidian division of $a_j$ by $q$.

\begin{theorem}\label{thm:neck}
When $r_j\in\{0,1,q-1\}$ for all $j$, then it is possible to choose for each $j$ such that $r_j\neq 0$, the thieves who get an additional bead of color $j$ and still have a fair $q$-splitting not requiring more than $(q-1)m$ cuts.
\end{theorem}

The following corollary is an immediate consequence of the previous theorem (and answered positively what was identified as a first interesting question by P{\'a}lv{\"o}lgyi).
\begin{corollary}
When there are three thieves, it is possible to choose for each $j$ such that $r_j\neq 0$, the thief (if $r_j=1$) or the two thieves (if $r_j=2$) who get an additional bead and still have a fair $3$-splitting not requiring more than $2m$ cuts.
\end{corollary}

It is proved by rounding in an appropriate way a fair splitting obtained by a continuous version of the splitting necklace theorem.

While in the splitting necklace theorem the cuts have to take place between the beads, this condition is relaxed in the ``continuous version'' of the splitting necklace theorem, also proved by Alon~\cite{alon1987splitting}. In this latter version, cuts are allowed to be located on beads, and  thieves may then receive fractions of beads. In this case, there always exists a {\em continuous fair $q$-splitting} for which each thief receives an amount of exactly $a_j/q$ beads of color $j$, with no more than $(q-1)m$ cuts.

Consider a continuous fair $q$-splitting and denote by $B_j$ the beads of color $j$ that are split between two or more thieves. Our result is obtained by showing that we can move the cuts located on beads in the $B_j$'s so that we reach a ``discrete'' fair $q$-splitting with the desired allocation. If $B_j=\varnothing$, then we already have a discrete fair splitting for the beads of color $j$. For each $j$ with $B_j\neq\varnothing$, we build a bipartite graph $G_j=(U_j,E_j)$ with the thieves on one side and the beads in $B_j$ on the other side. We put an edge between a thief $t$ and a bead $k$ if $t$ receives a part of $k$. For an edge $e=tk\in E_j$, let $u_e\in(0,1)$ be the amount of bead $k$ received by thief $t$. We have for all $k\in B_j$ and all $t\in[q]$ (we identify the thieves with the integers in $[q]$)
\begin{equation}\label{flow}
\sum_{e\in E_j}u_e=|B_j|,\quad\sum_{e\in\delta(k)}u_e=1,\quad\mbox{and}\quad\sum_{e\in\delta(t)}u_e=\alpha_{tj}+\frac {r_j} q\quad\mbox{for some integer $\alpha_{tj}\geq 0$,}
\end{equation} where $\delta(v)$ is the set of edges incident to a vertex $v$. Note that the degree of each thief-vertex $t$ in $G_j$ is at least $\alpha_{tj}+1$ and that the degree of each bead-vertex is at least $2$.

Changing the values of the $u_e$'s, while keeping them nonnegative and while satisfying the equalities~\eqref{flow}, leads to another continuous fair $q$-splitting with at most $(q-1)m$ cuts. The $u_e$'s form a flow. It is thus always possible to choose the continuous fair $q$-splitting in such a way that $G_j$ has no cycle for every $j$ (basic properties of flows). In the proofs below, $G_j$ will theferore always be assumed to be without cycle.
To get our result, we are going to select a subset $F$ of $E_j$ such that each bead-vertex is incident to exactly one edge in $F$. This subset of edges will encode an assignment of the beads in $B_j$ compatible with the already assigned beads (which does not increase the number of cuts), and leading to the desired allocation. 

For the proof of the case $r_j=1$, such a subset of edges is obtained as a special object of graph theory that we describe now.
Let $H=(V,E)$ be a bipartite graph and let $b:V\rightarrow\Z_+$. A {\em $b$-factor} is a subset $F\subseteq E$ such that each vertex $v\in V$ is incident to exactly $b(v)$ edges of $F$. There exists a $b$-factor if and only if each subset $X$ of $V$ spans at least $\sum_{v\in X}b(v)-\frac 1 2 \sum_{v\in V}b(v)$ edges, see~\cite[Corollary 21.4a]{schrijver2002combinatorial}.

\begin{lemma}\label{lem:r1}
When $r_j=1$, it is possible to move the cuts located on the beads of color $j$ and to get a discrete fair $q$-splitting for which we choose the thief getting the additional bead of color $j$. 
\end{lemma}

\begin{proof}
We have in this case $|B_j|=\sum_{t\in [q]}\alpha_{tj}+1$ (using~\eqref{flow}). For a thief $t$, define $b(t)=\alpha_{tj}$, except when $t$ is the thief chosen for the additional bead, in which case define $b(t)=\alpha_{tj}+1$. For each bead $k$, define $b(k)=1$. Consider a subset $X$ of vertices of $G_j$. Denote by $T$ the thief-vertices in $X$ and by $K$ the bead-vertices in $X$. The edges spanned by $X$ is $\delta(T)\setminus E[T:B_j\setminus K]$. We have $|\delta(T)|\geq\sum_{t\in T}(\alpha_{tj}+1)$ and $|E[T:B_j\setminus K]|\leq |T|+|B_j|-|K|-1$. To get this latter inequality, we use the fact that $G_j$ has no cycle. The number of edges spanned by $X$ is thus at least $|K|-|B_j|+1+\sum_{t\in T}\alpha_{tj}$. 

The quantity $\sum_{v\in X}b(v)-\frac 1 2 \sum_{v\in U_j}b(v)$ is at most $1+\sum_{t\in T}\alpha_{tj}+|K|-|B_j|$. According to the above mentioned result, there exists a $b$-factor in $G_j$.
\end{proof}

\begin{lemma}\label{lem:rq-1}
When $r_j=q-1$, it is possible to move the cuts located on the beads of color $j$ and to get a discrete fair $q$-splitting for which we choose the thief getting one bead of color $j$ less than the other thieves. 
\end{lemma}

\begin{proof}
Since $G_j$ is without cycle, its number of edges is at most $q-1+|B_j|$. The degree of each vertex in $B_j$ being at least $2$, it implies that $q-1+|B_j|\geq 2|B_j|$, and thus $|B_j|\leq q-1$. The fact that $r_j=q-1$ leads finally to $|B_j|=q-1$ (using~\eqref{flow}), which implies that $\alpha_{tj}=0$ for all $t\in[q]$. For any proper subset $T\subset [q]$ of distinct thieves, we have thus $\sum_{t\in T}\sum_{e\in\delta(t)}u_{e}=|T|-\frac {|T|}{q}$, which means that the size of the neighborhood of $T$ in $G_j$ is at least $\left\lceil|T|-\frac {|T|}{q}\right\rceil=|T|$. This latter equality holds because $|T|\leq q-1$. Hall's theorem ensures then that we can assign the $q-1$ beads in $B_j$ to any choice of $q-1$ thieves.
\end{proof}

\begin{proof}[Proof of Theorem~\ref{thm:neck}]
For the colors $j$ such that $r_j=0$, the original rounding procedure introduced by Alon, Moshkovitz, and Safra makes the job. For the colors $j$ such that $r_j\in\{1,q-1\}$, Lemmas~\ref{lem:r1} and~\ref{lem:rq-1} allow to conclude.
\end{proof}

Note that this approach may fail already when $q=4$ and $r_j=2$: if thieves $a$ and $b$ receive each half of a bead and thieves $c$ and $d$ receive each half of another bead, it is impossible to move the cuts so that both $a$ and $b$ are advantaged.

\subsection*{Acknowledgments} This work was done when the first author was visiting the Universit\'e Paris Est. He would like to acknowledge professor Fr\'ed\'eric Meunier for his generous support and hospitality.

\bibliographystyle{plain}
\bibliography{ColoredPath}

\end{document}